\documentclass[11pt,a4paper]{article}

\usepackage{geometry}
\usepackage{amssymb,amsbsy,amsmath,amscd,amsthm,amsfonts,MnSymbol}
\usepackage{cite}
\usepackage[utf8]{inputenc}
\usepackage{enumerate,hyperref}
\usepackage{dsfont}
\usepackage{graphicx}

\newtheorem{definition}{Definition}
\newtheorem{theorem}{Theorem}

\newtheorem{lemma}{Lemma}

\newtheorem{corollary}{Corollary}

\newcommand{\R}{\mathbb R}

\newcommand{\N}{\mathbb N}

\newcommand{\Exp}{\textrm{Exp}}

\newcommand*\diff{\mathop{}\!\mathrm{d}}

\newcommand*\dist{\mathop{}\!\mathrm{d}}

\providecommand{\keywords}[1]
{
  \small	
  \textbf{\textit{Keywords --- }} #1
}

\title{On the Continuity of Geodesically Convex Functions on Riemannian Manifolds}

\author{Victor-Emmanuel Brunel\thanks{Centre de Recherche en Economie et Statistique (CREST), Palaiseau, France}  \and Pierre Pansu \thanks{Université Paris-Saclay, CNRS, Laboratoire de Mathématiques d'Orsay, Orsay, France}}

\date{}

\begin{document}

\maketitle

\begin{abstract}
In this short note, we prove that all geodesically convex functions defined on a Riemannian manifold are continuous in the interior of their domain. This is a folklore result, but to the best of our knowledge, there is only one available proof, which is largely cited. However, it contains a significant gap, which we fill here. We also discuss extensions of this result beyond the Riemannian setting.

\end{abstract}


\section{Introduction}

In Euclidean spaces, convex functions are automatically continuous in the interior of their domain \cite[Theorem 10.1]{rockafellar1970convex}. This is a classical and simple result in finite dimensional convex geometry, and the analog statement is often granted to be true on Riemannian manifolds: Geodesically convex functions are automatically continuous in the interior of their domain. In fact, to the best of our knowledge, the only available proof of such a statement appears in \cite{udricste1977continuity}, and later in a book by the same author \cite[Chapter 3, Theorem 3.6]{udricste1994convex}. However, the proof has a gap since it implicitly assumes that a geodesically convex function is automatically locally bounded from above, which, then, renders the rest of the proof very easy. However, this fact is not justified and, as we found out, it requires some work. In these notes, we provide a complete proof of the statement. 

\section{Main result}

Let $(M,g)$ be a smooth Riemannian manifold of dimension $n\geq 1$. We denote by $\dist$ the distance inherited from $g$ on $M$. For $x\in M$, we denote by $T_xM$ the tangent space to $M$ at $x$ and by $g(x;\cdot,\cdot)$ the Euclidean product on $T_xM$ given by the metric $g$.

Given any $p,q\in M$, we call a geodesic from $p$ to $q$ any path $\gamma:[0,1]\to M$ satisfying $\gamma(0)=p, \gamma(1)=q$ and $\dist(\gamma(s),\gamma(t))=|s-t|\dist(x,y)$ for all $s,t\in [0,1]$. We let $\Gamma_{p,q}$ be the (possibly empty) collection of all geodesics from $p$ to $q$ (usually called \textit{minimizing geodesics} in Riemannian geometry). 

A subset $C\subset M$ is called geodesically convex if and only if for all $p,q\in C$ and all $\gamma\in\Gamma_{p,q}$, $\gamma([0,1])\subseteq C$. A function $f:M\to\R\cup\{\infty\}$ is called geodesically convex if and only for all $p,q\in M$, $\gamma\in \Gamma_{p,q}$ and $t\in [0,1]$, $f(\gamma(t))\leq (1-t)f(p)+tf(q)$. The domain of such a function is the set $f^{-1}(\R)$ -- it is necessarily geodesically convex. 

The main result of this note is the following.

\begin{theorem} \label{thm:Main}
	Any geodesically convex function on $M$ is locally Lipschitz in the interior of its domain. 
\end{theorem}

In particular, geodesically convex functions are continuous in the interior of their domain. We prove that a geodesically convex function is necessarily locally bounded above in the interior of its domain. The rest of the proof is skipped, as the remaining arguments already appear in \cite{udricste1977continuity}. In order to proceed, we introduce iterated barycenters and iterated barycentric hulls. 

Given a point $x_0\in M$, we call a \textit{totally normal neighborhood of} $x_0$ any open ball $U$ in $M$, centered at $x_0$, that is the diffeomorphic image of an open ball of $T_{x_0}M$ centered at $0$ by $\Exp_{x_0}$, that is geodesically convex and such that $\Gamma_{p,q}$ is a singleton for all $p,q\in U$. Existence of totally normal neighborhoods follows from classical results.

\begin{definition} \label{def:iterated_bary}
	Let $U$ be a totally normal neighborhood in $M$. We define inductively a sequence of functions $B_k:U^k\times [0,1]^{k-1}\to U$, $k=1,2,\ldots$ as follows. 
\begin{itemize}
	\item For all $p\in U$, let $B_1(p)=p$.
	\item For $k\geq 2$, $p_1,\ldots,p_k\in U$ and $t_2,\ldots,t_k\in [0,1]$, set $B_k(p_1,\ldots,p_k;t_2,\ldots,t_k)=\gamma(t_k)$, where $\gamma$ is the unique geodesic from $B_{k-1}(p_1,\ldots,p_{k-1};t_2,\ldots,t_{k-1})$ to $p_k$, where we use the convention that $B_{k-1}(p_1,\ldots,p_{k-1};t_2,\ldots,t_{k-1})=B_1(p_1)=p_1$ if $k=2$.
\end{itemize}
We call $B_k(p_1,\ldots,p_k;t_2,\ldots,t_k)$ the iterated barycenter of $p_1,\ldots,p_k$ with stepsizes $t_2,\ldots,t_k$. 
\end{definition}

For instance, $B_2(p_1,p_2;1/2)$ is the unique midpoint between $p_1$ and $p_2$. Note that the definition is non-ambiguous, since geodesics are unique in $U$. However, $B_k(p_1,\ldots,p_k;t_2,\ldots,t_k)$ depends on the ordering of the points $p_1,\ldots,p_k$ in general. 

Given $p_1,\ldots,p_k$ in a totally normal neighborhood of $M$, the set $\{B_k(p_1,\ldots,p_k;t_2,\ldots,t_k):0\leq t_2,\ldots,t_k\leq 1\}$ is called the iterated barycentric hull of $p_1,\ldots,p_k$ (which, again, depends on the ordering of $p_1,\ldots,p_k$ in general). When $k=d+1$, we call it an iterated simplex.

Our main lemma is the following.

\begin{lemma} \label{lem:Main}
	Every point of $M$ has a neighborhood that is included in an iterated simplex.
\end{lemma}

\begin{proof}
Let $x_0\in M$ and let $U$ be a totally normal neighborhood of $x_0$. Without loss of generality, assume that $U$ is a Euclidean ball centered at $x_0=0$, equipped with the Riemannian metric $g$ (indeed, otherwise, we would work in a local normal chart). For $p\in U$, we identify the tangent space $T_pU$ of $U$ at $p$ with $\R^d$. For all $p,q\in U$, we denote by $\gamma_{p,q}^{(1)}$ the unique geodesic from $p$ to $q$ relative to the metric $g$.

For $h\in[-1,1]$, let $g_h$ denote the metric on $U$ defined by $g_h(p;u,v)=g(hp;u,v)$, for all $p\in U$ and $u,v\in \R^d$. In particular, $g_1=g$ and $g_0$ is a Euclidean metric on $U$. 

Fix $h\in [-1,1]\setminus\{0\}$ and let $p,q\in U$. A geodesic from $p$ to $q$ relative to $g_h$ is a smooth path $\eta$ from $p$ to $q$ with minimum length 
\begin{align*}
	L(\eta) & = \int_0^1 g_h(\eta(t);\dot{\eta}(t),\dot{\eta}(t))^{1/2}\diff t \\
	& = \int_0^1 g(h\eta(t);\dot{\eta}(t),\dot{\eta}(t))^{1/2}\diff t \\
	& = |h|^{-1}\int_0^1 g(h\eta(t);h\dot{\eta}(t),h\dot{\eta}(t))^{1/2}\diff t
\end{align*}
which is uniquely minimized for $\eta=h^{-1}\gamma_{hp,hq}^{(1)}$. Therefore, there is a unique geodesic from $p$ to $q$ relative to $g_h$ and we denote it by $\gamma_{p,q}^{(h)}$. If $h=0$, since $(U,g_0)$ is Euclidean, there is also a unique geodesic from $p$ to $q$ relative to $g_0$: It is given by $\gamma_{p,q}^{(0)}(t)=(1-t)p+tq$, for all $t\in [0,1]$.

Since the map $(p,q,t)\in U\times U\times [0,1]\mapsto \gamma_{p,q}^{(1)}(t)$ is smooth (see \cite{lee2018introduction}), we obtain that for any fixed $p,q\in U$, the map $(h,t)\in [-1,1]\times [-1,1]\mapsto \gamma_{hp,hq}^{(1)}(t)$ is smooth, and since its value is $0$ when $h=0$, the map $(h,t)\in [-1,1]\times [-1,1]\mapsto \gamma_{hp,hq}^{(h)}(t)=h^{-1}\gamma_{hp,hq}^{(1)}(t)$ is also smooth. 

Now, fix $d+1$ affinely independent points $p_1,\ldots,p_{d+1}\in U$ such that $p_1+\ldots+p_{d+1}=0$ (so, $p_1,\ldots,p_{d+1}$ are the vertices of a Euclidean simplex containing $0$ in its interior). Define the map $F:[0,1]^d\times [-1,1]\to U$ by letting $F(t_2,\ldots,t_{d+1};h)$ be the iterated barycenter of $p_1,\ldots,p_{d+1}$ with stepsizes $t_2,\ldots,t_{d+1}$, relative to the metric $g_h$. That is, we define inductively $F_1(h)=p_1$, $F_2(t_2;h)=\gamma_{p_1,p_2}^{(h)}(t_2)$ and $F_k(t_2,\ldots,t_k;h)=\gamma_{F_{k-1}(t_2,\ldots,t_{k-1};h),p_k}^{(h)}(t_k)$ for $k=3,\ldots,d+1$ and we set $F(t_2,\ldots,t_{d+1};h)=F_{d+1}(t_2,\ldots,t_{d+1};h)$. Then, the map $F$ is smooth and it satisfies that 
$$F(1/2,1/3,\ldots,1/{d+1};0)=0,$$ 
by definition of $p_1,\ldots,p_{d+1}$. Moreover, an easy calculation yields that the Jacobian of $F$ with respect to ${\bf t}=(t_2,\ldots,t_{d+1})$ at ${\bf t}=(1/2,\ldots,1/(d+1))$ and $h=0$ is given by 
$$J_{\bf t}F(1/2,\ldots,1/(d+1);0)=\frac{1}{d+1}\left(\begin{matrix} 2(p_2-p_1)^\top \\ 3(p_3-\bar p_2)^\top \\ \vdots \\ (d+1)(p_{d+1}-\bar p_d)^\top \end{matrix} \right)$$
where $\bar p_j=(p_1+\ldots+p_j)/j$ for $j=2,\ldots,d+1$. Affine independence of $p_1,\ldots,p_{d+1}$ readily yields that this matrix is invertible. 

Now, by the implicit function theorem, we obtain the existence of $\varepsilon>0$, $\eta>0$ and a neighborhood $V$ of $(1/2,\ldots,1/(d+1))$ in $(0,1)^{d}$ such that for all $h\in (-\varepsilon,\varepsilon)$ and $x\in B(0,\eta)$, there is some $(t_2,\ldots,t_{d+1})\in V$ with $F(t_2,\ldots,t_{d+1};h)=x$. 
Fixing $h=\varepsilon/2$, this implies that $B(0,\eta)$ is included in the iterated barycentric hull of $p_1,\ldots,p_{d+1}$ relative to $g_h$, that is, $B(0,h\eta)$ is included in the iterated simplex spanned by $hp_1,\ldots,hp_{d+1}$. 
\end{proof}

\begin{proof}[Proof of Theorem~\ref{thm:Main}]
	Let $f:M\to\R\cup\{\infty\}$ be a geodesically convex function and let $x_0$ be in the interior of its domain. Let $U$ be a totally normal neighborhood of $x_0$ that is included in the domain of $f$. Let $p_1,\ldots,p_{d+1}\in U$ such that $x_0$ has a neighborhood $B(x_0,\eta)$ ($\eta>0$ that is included in the iterated simplex spanned by $p_1,\ldots,p_{d+1}$. Let $x\in B(x_0,\eta)$ and let $t_2,\ldots,t_{d+1}\in [0,1]$ such that $x=B_{d+1}(p_1,\ldots,p_{d+1};t_2,\ldots,t_{d+1})$. Using the notation of Definition~\ref{def:iterated_bary}, geodesic convexity of $f$ yields that  
$f(B_2(p_1,p_2;t_2))\leq (1-t_2)f(p_1)+t_2 f(p_2) \leq \max(f(p_1),f(p_2))$ and, by induction, 
\begin{align*}
	f(x) & = f(B_{d+1}(p_1,\ldots,p_{d+1};t_2,\ldots,t_{d+1})) \\
	& \leq (1-t_{d+1})f(B_d(p_1,\ldots,p_d);t_2,\ldots,t_d)+t_{d+1}f(p_{d+1}) \\
	& \leq (1-t_{d+1})\max(f(p_1),\ldots,f(p_d))+t_{d+1}f(p_{d+1}) \\
	& \leq \max(f(p_1),\ldots,f(p_{d+1})).
\end{align*}
This shows that $f$ is uniformly bounded from above on $B(x_0,\eta)$, by $\max(f(p_1),\ldots,f(p_{d+1}))$. The rest of the proof is classical. 
\end{proof}

As a consequence of Theorem~\ref{thm:Main}, we also obtain the following, where $\textrm{Vol}_g$ is the Riemannian volume measure of $(M,g)$.

\begin{corollary}
	A geodesically convex function on $M$ is differentiable $\textrm{Vol}_g$-almost everywhere.
\end{corollary}

\begin{proof}
	Let $(U,\phi)$ be any coordinate chart on $M$. That is, $U$ is an open subset of $M$ and $\phi: U\to V$ is a smooth diffeomorphism from $U$ to some open subset $V\subseteq\R^n$. Since $\phi^{-1}:V\to U$ is smooth, it is locally Lipschitz on $V$, and hence, so is $f\circ \phi^{-1}:V\to\R$. Therefore, Rademacher's theorem \cite[Theorem 3.1.6]{federer2014geometric} implies that $f\circ\phi^{-1}$ is differentiable $\lambda_n$-almost everywhere on $V$, where $\lambda_n$ is the Lebesgue measure on $\R^n$. In the coordinate chart $(U,\phi)$, the image of the Riemannian volume $\textrm{Vol}_g$ (that is, $\phi\#\textrm{Vol}_g$) has a positive density with respect to $\lambda_n$ (see, e.g., \cite[Proposition 2.41 (c)]{lee2018introduction}, hence, $f\circ\phi^{-1}$ is differentiable almost everywhere on $V$ with respect to that measure. In other words, $f$ is differentiable $\textrm{Vol}_g$-almost everywhere on $U$. Since this holds for any coordinate chart on $M$ and $M$ is second countable, this concludes the proof.
\end{proof}

\section{Beyond Riemannian manifolds}

The question of continuity of geodesically convex functions beyond Riemannian manifolds is open and seems to require more elaborate ideas. 

Fix a metric space $(M,\dist)$, which we assume to be geodesic for simplicity. That is, any two points in $M$ can be connected by at least one geodesic (using the same definition as in the previous section). Moreover, it is reasonable to assume that $(M,\dist)$ is locally compact, in order to exclude infinite dimensional linear normed spaces where there always exist non-continuous, convex (even linear) functions. The definition of geodesic convexity extends naturally to this setup. However, the following example shows that there are geodesically convex functions that are not continuous. 

Let $\tilde M=[0,1]\times \N^*$. Consider the equivalence relation given by $(x,m)\sim (y,n)$ if and only if $x=y=0$ or $x=y, m=n$. Now, let $M$ be the quotient of $\tilde M$ by this equivalence relation (that is, we glue countably many copies of $[0,1]$ at their origins). That is, $M$ can be viewed as a star with countably many branches that are copies of $[0,1]$, all glued together at one of their extremities. Define two metrics $\dist_1$ and $\dist_2$ on $M$ as follows. For $(x,m),(y,n)\in M$, set 
$$\dist_1((x,m),(y,n))=\begin{cases} |x-y|/m \mbox{ if } m=n \\ x/m+y/n \mbox{ otherwise} \end{cases}$$
and 
$$\dist_2((x,m),(y,n))=\begin{cases} |x-y| \mbox{ if } m=n \\ x+y \mbox{ otherwise.} \end{cases}$$

Then $(M,\dist_1)$ and $(M,\dist_2)$ are both geodesic. However, $(M,\dist_1)$ is locally compact (in fact, it is compact) and $(M,\dist_2)$ is not. In the sequel, we call $0$ the point $(0,n)\in M$, which does not depend on the choice of $n\geq 1$ by construction.

\begin{lemma}
The space $(M,\dist_1)$ is compact.
\end{lemma}

\begin{proof}
Let $((x_p,n_p))_{p\geq 1}$ be a sequence in $M$. If $(n_p)_{p\geq 1}$ has a constant subsequence, then it is clear that the sequence $((x_p,n_p))_{p\geq 1}$ has a converging subsequence obtained by further extracting a converging subsequence of $(x_p)_{p\geq 1}$. If $(n_p)_{p\geq 1}$ does not have a converging subsequence, then one may extract from it an increasing subsequence, which, after renumbering, we rename $n_1,n_2,\ldots$ without loss of generality. Then, the renumbered sequence $((x_p,n_p))_{p\geq 1}$ converges to $0$ since $0\leq \dist_1(x_p,n_p),0)=x_p/n_p\leq 1/n_p\xrightarrow[p\to\infty]{} 0$. 
\end{proof}

\begin{lemma}
The space $(M,\dist_2)$ is not locally compact.
\end{lemma}

\begin{proof}
We proceed by showing that $0$ has no neighborhood that is contained in a compact set. Let $U$ be a neighborhood of $0$ and let $\varepsilon>0$ be such that $U$ contains the ball $B$ centered at $0$ with radius $\varepsilon$, that is, the set $\{(x,n)\in M:0\leq x\leq \varepsilon,n\in\N^*\}$. Consider the sequence $((\varepsilon,p))_{p\geq 1}$, which is included in $B$. Yet, it does not have any converging subsequence, so $U$ cannot be contained in a compact set. 
\end{proof}

Let $f:M\to\R$ be the function defined by setting $f((x,n))=x$, for all $(x,n)\in M$. It is easy to check that $f$ is geodesically convex for both $\dist_1$ and $\dist_2$. However, it is continuous for $\dist_2$ but not for $\dist_1$. Indeed, while the sequence $((1,p))_{p\geq 1}$ converges to $0$ for $\dist_1$, $f((1,p))=1$ does not converge to $f(0)=0$. Of course, a similar example can be constructed on $(M,\dist_2)$, by defining the geodesically convex function $g$ by $g((x,n))=nx$ for all $(x,n)\in M$ and noting that while $(1/p,p)$ goes to $0$ for $\dist_2$, $g((1/p,p))=1$ does not go to $g(0)=0$.  

Now, let us notice that for the geodesically convex function $f$ that we constructed above, the only point of discontinuity (with respect to $\dist_1$) was the origin. In fact, even though $0$ is topologically an interior point of $(M,\dist_2)$, it is not a point where all geodesics can be extended by a fixed, positive amount. The following definition is similar to that of \textit{geodesic extension property} of \cite{lewis2023basic}.

\begin{definition}
	In a geodesic space $(M,\dist)$, we say that a point $x\in M$ is geodesically interior if there exists $\varepsilon>0$ such for all geodesics $\gamma$ with $x\in\gamma([0,1])$, there exists a geodesic $\tilde\gamma$ with $x\in\tilde\gamma([0,1])$ and $\min(\dist(x,\tilde\gamma(0)),\dist(x,\tilde\gamma(1)))\geq \varepsilon$. The family of points that are not geodesically interior is called the geodesic boundary of $(M,\dist)$.
\end{definition}

With this definition, $0$ is in the geodesic boundary of $(M,\dist_1)$ but not of $(M,\dist_2)$. Note that if $(M,g)$ is a Riemannian manifold with boundary, then its boundary is automatically included in the geodesic boundary (in the previous section, we only considered Riemannian manifolds without boundary). Indeed, given any boundary point $x$ of $M$, let $n\in T_xM$ be an outward pointing normal vector and let $\gamma$ be a geodesic in $M$ starting at $x$ and with $\dot{\gamma}(0)=-n$. Then, $\gamma$ cannot be extended to negative times. 

\cite[Lemma 4.3]{lewis2023basic} states that for a locally compact, geodesic space $(M,\dist)$, continuity of a geodesically convex function $f$ at a geodesically interior point $x$ is equivalent to $f$ being locally bounded from above in a neighborhood of $x$.

\bibliographystyle{plain}
\bibliography{biblio.bib}

\end{document}